%% file: arxiv1.tex
\documentclass[11pt]{amsart}

\usepackage{amsrefs}
\usepackage{amssymb}
\usepackage{fancyhdr}
\usepackage{bbm}
\usepackage{xcolor}
\usepackage{hyperref}
\hypersetup{
  colorlinks=true,
  linkcolor=blue,
  citecolor=blue
}
\usepackage{mleftright}
\usepackage{colonequals}

\newtheorem{theorem}{Theorem}
\newtheorem*{theorem*}{Theorem}
\newtheorem{lemma}[theorem]{Lemma}

\newtheorem{maintheorem}{Theorem}

\theoremstyle{definition}

\newtheorem*{definition*}{Definition}
\newtheorem*{lemma*}{Lemma}
\usepackage{graphicx}
\usepackage{pstricks, enumerate, pst-node, pst-text, pst-plot}

\numberwithin{equation}{section}
\numberwithin{theorem}{section}

\newcommand{\R}{\mathbb{R}}

\newcommand{\PP}{\mathbb{P}}

\usepackage{xparse}
\DeclareDocumentCommand\Pr{ m g }{\ensuremath{
    {   \IfNoValueTF {#2}
      {\mathbb{P}\mleft[{#1}\mright]}
      {\mathbb{P}\mleft[{#1}\middle\vert{#2}\mright]}%
    }
}}
\DeclareDocumentCommand\E{ m g }{\ensuremath{
    {   \IfNoValueTF {#2}
      {\mathbb{E}\mleft[{#1}\mright]}
      {\mathbb{E}\mleft[{#1}\middle\vert{#2}\mright]}%
    }
}}
\newcommand{\ind}[1]{{\mathbbm{1}_{\mleft\{{#1}\mright\}}}}
\newcommand{\embolden}[1]{\textbf {#1}}
\newcommand{\absf}[1]{|#1|}      
\newcommand{\abs}[1]{\left|#1\right|}

\begin{document}

\title[]{A Fourier approach to Levine's hat puzzle}

\author[]{Steven Heilman}
\address{University of Southern California}
\author[]{Omer Tamuz}
\address{California Institute of Technology}


\date{\today}

\begin{abstract}
We consider Lionel Levine's notorious hat puzzle with two players. Each player has a stack of hats on their head, and each hat is chosen independently to be either black or white. After observing only the other player's hats, players simultaneously choose one of their own hats. The players win if both chosen hats are black. In this note, we observe an upper bound on the probability of success, using Chang's lemma, a result in Boolean harmonic analysis.
\end{abstract}

\maketitle

\section{Introduction}
Fix $n \geq 1$, and let $X_1,\ldots,X_n,Y_1,\ldots,Y_n$ be independent identically distributed random variables with $\Pr{X_{1}=1}=\Pr{X_{1}=-1}=1/2$.  Denote $X = (X_1,\ldots,X_n)$, $Y=(Y_1,\ldots,Y_n)$. The two player version of Lionel Levine's notorious hat puzzle is to find $f,g \colon \{-1,1\}^n \to \{1,\ldots,n\}$ that maximize 
\begin{align}
    \label{eq:both_black}
    U_{f,g} \colonequals \Pr{X_{f(Y)} = Y_{g(X)} = 1}.
\end{align}
The interpretation is that $f,g$ are the players' strategies, and $X_i=1$ if the first player's $i$th hat is black, and $Y_i=1$ if the second player's $i$th hat is black. Hence the expression in \eqref{eq:both_black} is the probability that both players choose a black hat.

We would like to understand how high this probability can be. Let $[n]=\{1,\ldots,n\}$, let
\begin{align*}
    U_n = \max\{U_{f,g}\,:\, f,g \colon \{-1,1\}^n \to [n]\},
\end{align*}
and observe that $U_{n+1} \geq U_n$ for all $n\geq1$. We would like to know $U \colonequals \lim_{n\to\infty}  U_n = \sup_{n\geq1} U_n$.

It was conjectured in \cites{buhler22,alon23} that $U = 7/20=0.35$. In \cites{buhler22} it was shown that $U \geq 7/20$, and, using a simple argument, that $U \leq 3/8=0.375$. Using a computer assisted proof, they improved this upper bound to $U \leq 81/224 \approx 0.3616$. As far as we know, this is the best known upper bound.

In this note we take a Fourier approach to this question. While we do not improve the best known upper bound, our technique yields that $U \leq 0.37406$, without computer assistance (Theorem~\ref{thm:main}). A longer, additional argument leads to the slightly improve bound $U\leq .37193$.

We can write \eqref{eq:both_black} as
\begin{align*}
    U_{f,g} = \E{\frac{1}{2}(X_{f(Y)}+1) \cdot \frac{1}{2}(Y_{g(X)}+1)},
\end{align*}
and since $X$ and $Y$ are independent, $\E{X_{f(Y)}}=0$, and so $U_{f,g}$ can be rewritten as
\begin{align*}
        U_{f,g} = \frac{1}{4}\E{X_{f(Y)} \cdot Y_{g(X)}}+\frac{1}{4}.
\end{align*}
Maximizing this expression over $f,g$ is of course equivalent to maximizing
\begin{align}
    \label{eq:match}
        W_{f,g} \colonequals \E{X_{f(Y)} \cdot Y_{g(X)}}.
\end{align}
This corresponds to a game in which the two players win a dollar if the colors of their chosen hats match, and lose a dollar if they mismatch. This equivalent form will be more convenient for our purposes.

To see the connection to boolean harmonic analysis, given $f,g \colon \{-1,1\}^{n} \to [n]$ and $1\leq i,j\leq n$, denote
\begin{align}\label{zero0}
   F_{ij} &= \E{Y_i \cdot \ind{f(Y)=j}} \nonumber\\
   G_{ij} &= \E{X_i \cdot \ind{g(X)=j}}.
\end{align}
A first observation (Lemma~\ref{olem}) is that
\begin{align}
   \label{eq:FG} \E{X_{f(Y)}Y_{g(X)}}=\sum_{i,j=1}^{n}F_{ij}G_{ji}.
\end{align}
The advantage of this representation is that $F_{ij}$ depends only on $f$, and $G_{ij}$ depends only on $g$. Moreover, these are simply the level one Fourier coefficients of the level sets of $f$ and $g$. Indeed, given a function $\varphi \colon \{-1,1\}^n \to \R$, its Fourier transform $\widehat \varphi \colon 2^{[n]} \to \R$ is given by
\begin{align}
    \label{eq:fourier}
    \widehat\varphi(S) = \E{\varphi(Y) \cdot \prod_{i \in S}Y_i }.
\end{align}
When $S = \{i\}$ is a singleton we write $\widehat \varphi(i) \colonequals \widehat \varphi(\{i\})$. Then $F_{ij} \stackrel{\eqref{zero0}}{=} \widehat \varphi_j(i)$, where $\varphi_j = \ind{f=j}$.  

A useful consequence of \eqref{eq:FG}, which we will exploit, is that it implies that $W_{f,g}$ can be upper bounded using the Cauchy-Schwarz inequality, given estimates of $F_{ij}$ and $G_{ij}$:
\begin{align*}
    \E{X_{f(Y)}Y_{g(X)}} \leq \sqrt{
    \sum_{i,j=1}^{n}F_{ij}^2
    \sum_{i',j'=1}^{n}G_{i'j'}^2
    }.
\end{align*}

\section{Boolean harmonic analysis}
Fix $n \geq 1$, and consider the vector space $\R^{\{-1,1\}^n}$ of functions $h \colon \{-1,1\}^n \to \R$. 
Define an inner product on our vector space by 
\begin{align*}
    \langle h , u \rangle := \E{h(X) \cdot u(X)},\qquad\text{for all}\, h,u\in \R^{\{-1,1\}^n}.
\end{align*}

We will use two important results. The first is Plancherel's Theorem, which states that given $h,u \colon \{-1,1\}^n \to \R$, 
\begin{align}
    \label{eq:plancherel}    
    \langle h,u \rangle = \sum_{S\subseteq[n]} \widehat h(S) \widehat u(S).
\end{align}

The second (and less trivial) result is Chang's inequality.
\begin{lemma}[\embolden{Chang's Inequality}, {\cite{impa14},\cite[page 126]{odonnell14}}]\label{changlemma}
\hfill\\
Let $h\colon\{-1,1\}^{n}\to\{0,1\}$, and denote $\alpha\colonequals\Pr{h(X)=1}$.    Then
\begin{align}
    \label{eq:chang}
  \sum_{k=1}^{n}\absf{\widehat{h}(k)}^{2}\leq 2\alpha^{2}\log(1/\alpha).  
\end{align}
\end{lemma}
Here $\log$ denotes the natural logarithm.

The bound of \eqref{eq:chang} is tight for small $\alpha$ tending to zero. Another bound is useful for large $\alpha$:
\begin{lemma}\label{olem2}
Let $h\colon\{-1,1\}^{n}\to\{0,1\}$, and denote $\alpha\colonequals\Pr{h(X)=1}$.    Then
\begin{align*}
  \sum_{k=1}^{n}\absf{\widehat{h}(k)}^{2}\leq \alpha/2.  
\end{align*}
\end{lemma}
\begin{proof}
    Let $u(x) = h(-x)$, and note that $\widehat{u}(S) = (-1)^{|S|}\widehat h(S)$. Hence, by Plancherel \eqref{eq:plancherel},
    \begin{align*}
        \langle h, u \rangle = \sum_{S\subseteq[n]} (-1)^{|S|} \absf{\widehat h(S)}^2.
    \end{align*}
    It follows (again by Plancherel) that
    \begin{align}
        \label{eq:hdif}
        \langle h, h-u \rangle = 2\sum_{S\subseteq[n]\colon|S|\text{ odd}} \absf{\widehat h(S)}^2.
    \end{align}
    Since $\absf{h}^2=h$, 
    \begin{align*}
        \alpha = \E{h(X)} = \E{h(X) \cdot h(X)} = \langle h, h\rangle.
    \end{align*}
    Additionally, $\langle h,u\rangle\geq0$ since $h,u\geq0$.  Thus, 
    \begin{align*}
        \sum_{S\subseteq[n]\colon|S|\text{ odd}} \absf{\widehat h(S)}^2 \stackrel{\eqref{eq:hdif}}{\leq} \frac{1}{2}\big(\langle h,h\rangle - \langle h,u\rangle\big) 
        \leq \langle h,h\rangle/2\leq  \alpha/2.
    \end{align*}
    Finally, since $S=\{k\}$ is of odd size, we have that 
    \begin{align*}
        \sum_{k=1}^{n}\absf{\widehat{h}(k)}^{2} \leq \sum_{S\subseteq[n]\colon|S|\text{ odd}} \absf{\widehat h(S)}^2 \leq\alpha/2.
    \end{align*}
\end{proof}

Putting these lemmas together we have that for any $h \colon \{-1,1\}^n \to \{0,1\}$ and $\alpha = \Pr{h(X)=1}$ it holds that
\begin{align}
    \label{eq:chang3}
      \sum_{k=1}^{n}\absf{\widehat{h}(k)}^{2}\leq \min\{2\alpha^{2}\log(1/\alpha),\,\alpha/2\}.
\end{align}

\section{Proofs}
Recall that
\begin{align*}
    U = \sup_{n\geq1} \max_{f,g \colon \{-1,1\}^n\to[n]}\E{\frac{1}{2}(X_{f(Y)}+1) \cdot \frac{1}{2}(Y_{g(X)}+1)}.
\end{align*}
Our main result is the following.
\begin{maintheorem}
\label{thm:main}
$U \leq 0.37406$.
\end{maintheorem}
In fact, with some additional effort we improve this bound to $U\leq .37193$ in Section \ref{secsmall}.

It will be more convenient to study
\begin{align*}
    W = 4U-1 = \sup_{n\geq1} \max_{f,g \colon \{-1,1\}^n\to[n]}\E{X_{f(Y)} \cdot Y_{g(X)}}.
\end{align*}
In terms of $W$, Theorem~\ref{thm:main} states that $W \leq 0.496235$. We prove this theorem in the remainder of this note.

We begin by showing \eqref{eq:FG}.
\begin{lemma}\label{olem}
For any $f,g \colon\{-1,1\}^{n} \to \{1,\ldots,n\}$ we have
$$\E{X_{f(Y)}Y_{g(X)}}=\sum_{i,j=1}^{n}F_{ij}G_{ji}.$$
\end{lemma}
\begin{proof}
Using the independence of $X$ and $Y$,
\begin{align*}
\E{X_{f(Y)}\cdot Y_{g(X)}}
  &= \sum_{i,j=1}^{n}\E{X_j \cdot Y_i\cdot \ind{g(X)=i} \cdot \ind{f(Y)=j}}\\
  &= \sum_{i,j=1}^{n}\E{X_j \cdot \ind{g(X)=i}} \cdot \E{Y_i \cdot
    \ind{f(Y)=j}}\\
  &\stackrel{\eqref{zero0}}{=} \sum_{i,j=1}^{n} F_{ij}G_{ji}.
\end{align*}
\end{proof}

In the next lemma we show that $W \leq 1/2$ (corresponding to $U \leq 3/8$, as in \cite{buhler22}), and that this bound can be furthermore improved if we can restrict the probability that $f$ (or $g$) is large.
\begin{lemma}\label{lemma7}
Let $k>0$ and let $0<\epsilon<1/4$.  Assume that 
\begin{equation}\label{three1}
\Pr{g(X)\geq k}\leq\epsilon.
\end{equation}
Then
$$\E{X_{f(Y)}Y_{g(X)}}\leq\frac{1}{2}-(1-4\epsilon)2^{-k+1}.$$
\end{lemma}
\begin{proof}
Recall our notation $W_{f,g} = \E{X_{f(Y)}Y_{g(X)}}$.  Since $X$ is independent of $Y$, since $X$ and $-X$ have the same distribution, and since $Y$ and $-Y$ have the same distribution, $(X,Y)$ has the same joint distribution as $(X,-Y)$, $(-X,Y)$ or $(-X,-Y)$. Hence
\begin{align*}
  W_{f,g} = \E{-X_{f(-Y)}Y_{g(X)}} = \E{-X_{f(Y)}Y_{g(-X)}}  = \E{X_{f(-Y)}Y_{g(-X)}}.
\end{align*}
Summing these and rearranging yields
\begin{align*}
  4W_{f,g} = \E{(X_{f(Y)}-X_{f(-Y)})(Y_{g(X)}-Y_{g(-X)})}.
\end{align*}

Let $B_Y$ be the event that $Y_1 = Y_2 = \cdots = Y_{k-1}$, and note that $\Pr{B_Y} = 2^{-(k-2)}$. Write
\begin{align}
  \label{eq:Bk-bound}
  4W_{f,g} =
  &\E{(X_{f(Y)}-X_{f(-Y)})(Y_{g(X)}-Y_{g(-X)})\ind{B_Y}}\nonumber\\
  &+ \E{(X_{f(Y)}-X_{f(-Y)})(Y_{g(X)}-Y_{g(-X)})\ind{B_Y^c}}.
\end{align}
Consider the first term on the right hand side. Since $\Pr{g(-X) \geq k} = \Pr{g(X) \geq k} \leq \epsilon$, we have that $\Pr{g(X) \geq k \text{ or } g(-X) \geq k} \leq 2\epsilon$, by the union bound. Now, if $g(X) < k$ and $g(-X) < k$ then, under the event $B_Y$,  $Y_{g(X)} = Y_{g(-X)}$. Hence
\begin{align*}
  (Y_{g(X)}-Y_{g(-X)}) \cdot \ind{B_Y} = (Y_{g(X)}-Y_{g(-X)})\cdot\ind{g(X) \geq k \text{ or } g(-X)\geq k}\cdot\ind{B_Y}.
\end{align*}
We thus have by independence of $X$ and $Y$ that
\begin{align*}
  \lefteqn{\E{(X_{f(Y)}-X_{f(-Y)})(Y_{g(X)}-Y_{g(-X)})\ind{B_Y}}}\\
  &= \E{(X_{f(Y)}-X_{f(-Y)})\cdot(Y_{g(X)}-Y_{g(-X)})\cdot\ind{g(X) \geq k \text{ or } g(-X)\geq k}\cdot\ind{B_Y}}\\
  &\leq \E{4\cdot\ind{g(X) \geq k \text{ or } g(-X)\geq k}\cdot\ind{B_Y}}\\
  &= 4\Pr{g(X) \geq k \text{ or } g(-X)\geq k,B_Y}\\
  &=4\Pr{g(X) \geq k \text{ or } g(-X)\geq k}\Pr{B_Y}
  \\
  &\leq 8\epsilon 2^{-(k-2)}.
\end{align*}

Consider now the second term on the right hand side of \eqref{eq:Bk-bound}. Noting that $\Pr{X_{f(Y)}=X_{f(-Y)}}{Y=y} \geq 1/2$ for any $y$, we get that
\begin{align*}
  \E{(X_{f(Y)}-X_{f(-Y)})(Y_{g(X)}-Y_{g(-X)})\ind{B_Y^c}} \leq 2\Pr{B_Y^c} = 2(1-2^{-(k-2)}).
\end{align*}

Substituting these two bounds back into \eqref{eq:Bk-bound} and diving by $4$, we get
\begin{align*}
  W_{f,g} \leq \frac{1}{4}\left(8\epsilon 2^{-(k-2)} + 2(1 - 2^{-(k-2))}\right) = \frac{1}{2}-(1-4\epsilon)2^{-k+1}.
\end{align*}
\end{proof}

Finally, let us apply the upper bound \eqref{eq:chang3} to the case not covered by Lemma \ref{lemma7}.

\begin{lemma}\label{lemma6}
Let $\Omega_{f}\subseteq[n]$ be 8 indices (or $n$ indices if $n<8$) $i\in[n]$ with the largest values of $\Pr{f(X)=i}$.  Similarly define $\Omega_{g}\subseteq[n]$.  Assume that
$$\Pr{f(X)\notin\Omega_{f}}\geq .009079 \qquad\mathrm{and}\qquad \Pr{g(X)\notin\Omega_{g}}\geq .009079.$$
Then
$$\E{X_{g(Y)}Y_{f(X)}}<.4962357.$$
\end{lemma}
\begin{proof}
For any $j \in [n]$, define $\alpha_{j}\colonequals\Pr{f(X)=j}$, and recall that $F_{ij} \colonequals \widehat{\ind{f=j}}(i)$. We prove the claim by bounding  the quantity \begin{equation}\label{mmeq}
\|F\| \colonequals \sqrt{\sum_{j=1}^{n}\sum_{i=1}^{n}\absf{F_{ij}}^{2}}.
\end{equation}
Indeed, by \eqref{eq:FG} and the Cauchy-Schwarz inequality we have that 
$$
\E{X_{g(Y)}Y_{f(X)}} \leq \|F\|\cdot\|G\|,
$$
and so it suffices to show that $\|F\|^2,\|G\|^2 \leq 0.4962357$.

Without loss of generality, we can apply a permutation to $[n]$ so that $\Pr{f(X)=1}\geq \Pr{f(X)=2} \geq \cdots \geq \Pr{f(X)=n}$, and the assumption on $f$ becomes
$$\Pr{f(X)>8}\geq.009079.$$
Since we have already applied the Cauchy-Schwarz inequality, we can apply another permutation to $[n]$ to similarly get this conclusion for the function $g$.  Below we will only bound $\|F\|$, the case of $g$ being identical.

Denote $\alpha'\approx .116101$ as the unique $0<\alpha<1/2$ such that $2\alpha^{2}\log(1/\alpha)=\alpha/2$.  Since $(F_{ij})_{i=1}^{n}$ are the first level Fourier coefficients of $\ind{f=j}$,  we can use \eqref{eq:chang3} to upper bound $\|F\|^2$ by 
\begin{equation}\label{meq}
\sum_{j=1}^{n}\Big(2\alpha_{j}^{2}\log(1/\alpha_{j})
1_{\{\alpha_{j}<\alpha'\}} +\quad \frac{\alpha_{j}}{2}1_{\{\alpha_{j}\geq\alpha'\}}\Big).
\end{equation}
Subsequently, and given the lemma hypotheses, we can bound $\|F\|^2$ by the maximum of \eqref{meq} subject to the constraints: \begin{equation}\label{meqcon}
\sum_{j=9}^{n}\alpha_{j}\geq.009
,\quad \alpha_{1}\geq\alpha_{2}\geq\cdots\geq\alpha_{n}
,\quad\sum_{j=1}^{n}\alpha_{j}=1.
\end{equation}
These constraints imply $\alpha_{9}\leq 1/9$.  Note also that $1/9<\alpha'$.  The function $z(\alpha)= 2\alpha^{2}\log(1/\alpha)
1_{\{\alpha<\alpha'\}} +(\alpha/2)1_{\{\alpha\geq\alpha'\}}$ is convex for $0<\alpha<\alpha'$ and linear for $\alpha'\leq\alpha<1$.  So, if $\alpha_{1},\alpha_{2},\ldots$ is a maximum of \eqref{meq} subject to the constraints \eqref{meqcon}, we may assume that $\alpha_{1},\ldots,\alpha_{8}\geq\alpha'$, $\alpha_{j}=0$ for all $j>9$, and therefore $\alpha_{9}\leq 1-8\alpha'\leq.0712$, since $z(\alpha)$ is strictly convex for $0<\alpha<1/9$, and all $\alpha_{j}$ with $j\geq9$ are in this region.  The function $\alpha\mapsto (1-\alpha)/2 + 2\alpha^{2}\log(1/\alpha)$ is convex for $0<\alpha<.0712$, so \eqref{meq} is at most $4\alpha'+2(1-8\alpha')^{2}\log(1/(1-8\alpha'))\leq.4912$, or
\begin{flalign*}
&2\alpha_{9}^{2}\log(1/\alpha_{9})
+\sum_{j=1}^{8}\frac{\alpha_{j}}{2}\\
&\qquad=2(.009079)^{2}\log(1/.009079) + (1-.009079)/2
\approx .49626356.
\end{flalign*}
%
%

\end{proof}

Combining with Lemma \ref{lemma6} then gives an unconditional upper bound which is worse than that of \cite{buhler22}, albeit without computer assistance.

\begin{proof}[Proof of Theorem \ref{thm:main}]

Let $f,g\colon\{-1,1\}^{n}\to[n]$.  Let $\Omega_{f}\subseteq[n]$ be 8 indices (or $n$ indices if $n<8$) $i\in[n]$ with the largest values of $\Pr{f(X)=i}$.  Similarly define $\Omega_{g}\subseteq[n]$.

Suppose for now that $\Pr{g(X)\notin\Omega_{g}}<.009079$.  After applying a permutation to $[n]$, we may assume that $\Pr{g(X)>8}<.009079$.  Then we can apply Lemma \ref{lemma7} with $k=9$ and $\epsilon=.009079$ to get
$$\E{X_{g(Y)}Y_{f(X)}}\leq\frac{1}{2}-(1-4(.009079))2^{-8}\approx 0.4962356.$$
As discussed after the statement of Theorem \ref{thm:main}, this proves Theorem \ref{thm:main}, since $(1+.4962357)/4\leq .37406$.  We get the same conclusion from Lemma \ref{lemma7} if $\Pr{f(X)\notin\Omega_{f}}<.009079$.  The only remaining case to consider is:
$$\Pr{f(X)\notin\Omega_{f}}\geq .009079\quad\text{and}\quad\Pr{g(X)\notin\Omega_{g}}\geq .009079.$$
In this case, Lemma \ref{lemma6} concludes the proof.
\end{proof}

\section{A Small Improvement}\label{secsmall}.

\begin{figure}[ht!]
\centering
\def\svgwidth{.7\textwidth}
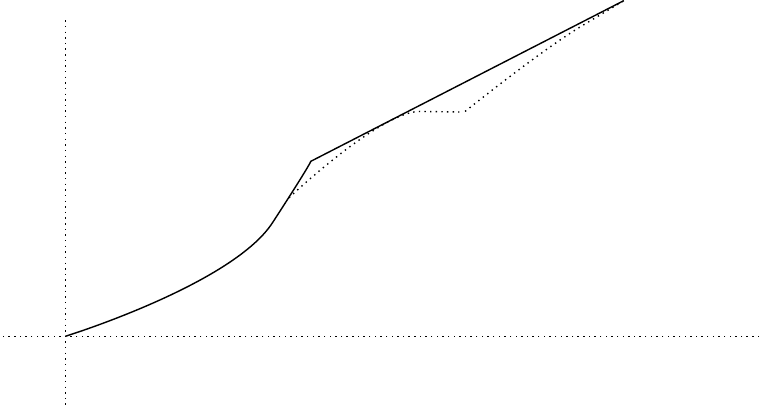
\caption{The solid line shows the upper bound on squared first order Fourier coefficients from \eqref{eq:chang3}.  The dotted line shows the better (smaller) upper bound from Lemma \ref{newlemma}.  No improvement to \eqref{eq:chang3} is possible when $x=\Pr{A}\in\{1/4,1/2\}$.  
}.
\label{figone}
\end{figure}

The proof of Lemma \ref{olem2} gives a stronger inequality than the statement of Lemma \ref{olem2}, namely: if $h\colon\{-1,1\}^{n}\to\{0,1\}$, and if $\alpha\colonequals\Pr{h(X)=1}$, then
\begin{equation}\label{bonus}
\sum_{k=1}^{n}|\widehat{h}(k)|^{2}
\leq\frac{\alpha}{2}-\sum_{|S|\geq3\text{ odd}}|\widehat{h}(S)|^{2}.
\end{equation}
So, lower bounding the last quantity leads to an improved bound in Lemma \ref{olem2}, which then leads to better constants in Chang's inequality and the proof of the Main Theorem \ref{thm:main}.  This lower bound is stated in Lemma \ref{newlemma} below.

There are exactly two cases where the last term in \eqref{bonus} is zero, namely when $h=1$ on a subcube of measure $1/2$ or $1/4$.  This follows since this term is zero when $(h+h(-\cdot))^{3}=h+h(-\cdot)$, which implies that $h+h(-\cdot)$ has at most two nonzero terms in its first order coefficients.  The restriction that $h\in\{0,1\}$ can then be used to show that these nonzero coefficients have value either $1$ or $1/2$.  One can ``add epsilons'' to this argument to obtain a worse version of Lemma \ref{newlemma}.  We instead prove Lemma \ref{newlemma} using other elementary methods.

\newcommand{\vnorm}[1]{\left\|#1\right\|}    

In this section we denote $\vnorm{f}\colonequals(\E{f}^{2})^{1/2}$ for any $f\colon\{-1,1\}^{n}\to\R$.

\begin{lemma}\label{newlemma}
Let $A\subset\{-1,1\}^{n}$ with $A\cap (-A)=\emptyset$. 
If $$\sum_{S\subseteq[n]\colon \abs{S}\geq3\,\mathrm{odd}}|\widehat{\ind{A}}(S)|^{2}< .000422,$$
then
\begin{flalign*}
\sum_{S\subseteq[n]\colon \abs{S}\geq3\,\mathrm{odd}}|&\widehat{\ind{A}}(S)|^{2}
\geq\frac{1}{4}\min\Big(
(2\Pr{A})^{1/.69}38^{-1/.69}
,\\
&\qquad
\Big[\frac{\sqrt{4+4(34)(1/2 - 2\Pr{A})} - 2}{68}\Big]^{2/.69}\ind{\Pr{A}<1/4}
\\
&\qquad\quad+\Big[\frac{\sqrt{4-4(40)(1/2 - 2\Pr{A})} - 2}{80}\Big]^{2/.69}\ind{\Pr{A}\geq1/4}
,\\
&\qquad
\Big[\frac{\sqrt{16+4(32)(1 - 2\Pr{A})} - 4}{64}\Big]^{2/.69}1_{\Pr{A}<1/2}
\Big)
\end{flalign*}
\end{lemma}

%
\begin{proof}
Let $A\subset\{-1,1\}^{n}$ such that $A\cap(-A)=\emptyset$ and let $g\colonequals (\ind{A} - \ind{-A})$ so that $g\colon\{-1,1\}^{n}\to\{-1,0,1\}$.  Let $f\colon\{-1,1\}^{n}\to\R$ be a linear function that minimizes
$$\vnorm{g-\ell}^{2}$$
over all linear functions $\ell\colon\{-1,1\}^{n}\to\R$.  That is, $f$ is the closest possible linear approximation to $g$.  Write
\begin{equation}\label{fdef}
f(x)=\sum_{i=1}^{n}a_{i}x_{i},\qquad\forall\, x\in\{-1,1\}^{n}.
\end{equation}
for some $a_{1},\ldots,a_{n}\in\R$.  
Define
\begin{equation}\label{epsdef}
\epsilon\colonequals\vnorm{f-g}.
\end{equation}

For any $1\leq i\leq n$, define
$$
f_{i}(x)\colonequals 
f(x_{1},\ldots,x_{i-1},-x_{i},x_{i+1},\ldots,x_{n}),
\qquad\forall\, x\in\{-1,1\}^{n}.
$$
Fix $0<c<1$ to be chosen later.  Since $g\in\{-1,0,1\}$,
\begin{equation}\label{ten2}
\begin{aligned}
&\Pr{f\notin(\{-1,0,1\}+[-\epsilon^{c},\epsilon^{c}])}\\
&\qquad\leq 
\Pr{\abs{f-g}>\epsilon^{c}}\leq \vnorm{f-g}^{2}/\epsilon^{2c}\stackrel{\eqref{epsdef}}{=}\epsilon^{2(1-c)}.
\end{aligned}
\end{equation}

Similarly, $\forall$ $1\leq i\leq n$, $g_{i}(x)\colonequals g(x_{1},\ldots,x_{i-1},-x_{i},x_{i+1},\ldots,x_{n})$ satisfies $\vnorm{f_{i}-g_{i}}\stackrel{\eqref{epsdef}}{=}\epsilon$, and
$$
\Pr{f_{i}\notin(\{-1,0,1\}
+[-\epsilon^{c},\epsilon^{c}])}
\leq\epsilon^{2(1-c)}.
$$
Combining with \eqref{ten2}, we get
$$
\Pr{f+f_{i}\notin\{-2,-1,0,1,2\}
+[-2\epsilon^{c},2\epsilon^{c}]}
\leq2\epsilon^{2(1-c)}.
$$
By \eqref{fdef}, $f+f_{i}=2a_{i}x_{i}$, so
$$
\Pr{a_{i}x_{i}\notin\{-1,-1/2,0,1/2,1\}
+[-\epsilon^{c},\epsilon^{c}]}
\leq2\epsilon^{2(1-c)},\quad\forall\,1\leq i\leq n.
$$
That is, if $\epsilon<(1/2)^{\frac{1}{2(1-c)}}$, we have
\begin{equation}\label{aieq}
a_{i}\in
\{-1,-1/2,0,1/2,1\}
+[-\epsilon^{c},\epsilon^{c}],
\qquad\forall\,1\leq i\leq n.
\end{equation}
%
%
%

We split into cases by how many indices $i$ satisfy $|a_{i}|>\epsilon^{c}$ in \eqref{aieq}.

\textbf{Case 1}.  $\exists$ $1\leq i\leq n$ with $\abs{a_{i}}\in 1+[-\epsilon^{c},\epsilon^{c}]$.  

\textbf{Case 2}.  All $1\leq i\leq n$ satisfy
$$a_{i}\in[-\epsilon^{c},\epsilon^{c}].$$
Fix $1\leq i\leq n$.  From the Berry-Esseen CLT \cite[Equation (2)]{janisch24}, if $G$ is a mean zero Gaussian random variable with variance $1$, then
$$
\sup_{t\in\R}\Big|\Pr{\sum_{i=1}^{n}a_{i}X_{i}<t} -\Pr{G\sqrt{\sum_{i=1}^{n}a_{i}^{2}}<t} \Big|
\leq.56\frac{\sum_{i=1}^{n}|a_{i}|^{3}}{(\sum_{i=1}^{n}a_{i}^{2})^{3/2}}.
$$
In particular, choosing $t=\epsilon^{c}$ and $t=1-\epsilon^{c}$, and using $\abs{a_{i}}\leq\epsilon^{c}$ for all $1\leq i\leq n$,
$$
\Big|\Pr{\epsilon^{c}<f(X)<1-\epsilon^{c}} -\Pr{\epsilon^{c}<G\sqrt{\sum_{i=1}^{n}a_{i}^{2}}<1-\epsilon^{c}} \Big|
\leq\frac{1.12\epsilon^{c}}{\sqrt{\sum_{i=1}^{n}a_{i}^{2}}}.
$$
If $\epsilon<(1/2)^{1/[2(1-c)]}$, then \eqref{ten2} (and that $f,f(-\cdot)$ have the same distribution to get an extra $1/2$ factor on the right) imply that
\begin{equation}\label{eq:berry}
\int_{\epsilon^{c}/\sqrt{\sum_{i=1}^{n}a_{i}^{2}}}^{(1-\epsilon^{c})/\sqrt{\sum_{i=1}^{n}a_{i}^{2}}}\frac{e^{-z^{2}/2}dz}{\sqrt{2\pi}} 
\leq\epsilon^{2(1-c)}/2+\frac{1.12\epsilon^{c}}{\sqrt{\sum_{i=1}^{n}a_{i}^{2}}}.
\end{equation}
If $\epsilon^{c}/\sqrt{\sum_{i=1}^{n}a_{i}^{2}}<1/6$, then we get
$$
\int_{1/6}^{1-1/6}\frac{e^{-z^{2}/2}dz}{\sqrt{2\pi}} 
\leq\epsilon^{2(1-c)}/2+1.12/6.
$$

That is,
$$
.231488
\leq\epsilon^{2(1-c)}/2+.18666,
$$
which implies that $\epsilon^{2(1-c)}>.089656$.

In other words, if $\epsilon<.089656^{1/[2(1-c)]}$, we have shown that $\epsilon^{c}/\sqrt{\sum_{i=1}^{n}a_{i}^{2}}<1/6$ does not occur, i.e.
$$\sum_{i=1}^{n}a_{i}^{2}<36\epsilon^{2c}.$$
That is,
$$\vnorm{f}^{2}\stackrel{\eqref{fdef}}{\leq} 36\epsilon^{2c}.$$

\textbf{Case 3}.  Case 1 does not occur, and $\exists$ exactly two $1\leq i<j\leq n$ with 
$\abs{a_{i}},\abs{a_{j}}\in 
1/2+[-\epsilon^{c},\epsilon^{c}]$.  Then all other coefficients are at most $\epsilon^{c}$ in absolute value, so applying Case 2 to the function $f-(a_{i}x_{i}+a_{j}x_{j})$ gives
$$\vnorm{f-(a_{i}x_{i}+a_{j}x_{j})}^{2}<36\epsilon^{2c}.$$
Compared to Case 2, instead of using \eqref{ten2} in \eqref{eq:berry} we use
\begin{flalign*}
&\Pr{f-(a_{i}x_{i}+a_{j}x_{j})\notin(\{-2,-1,0,1,2\}+[-\epsilon^{c},\epsilon^{c}])}\\
&\qquad\leq 
\Pr{\abs{f-(a_{i}x_{i}+a_{j}x_{j})-[g-(a_{i}x_{i}+a_{j}x_{j})]}>\epsilon^{c}}\\
&\qquad\leq \vnorm{f-g}^{2}/\epsilon^{2c}\stackrel{\eqref{epsdef}}{=}\epsilon^{2(1-c)}.
\end{flalign*}

\textbf{Case 4}.  Case 1 does not occur, and $\exists$ exactly one, three or four indices $1\leq i\leq n$ with
$\abs{a_{i}}\in 
1/2+[-\epsilon^{c},\epsilon^{c}]$.  This case is similar to Case 3, but (e.g. in the case of exactly one index) we get
$$\vnorm{f-a_{i}x_{i}}^{2}<36\epsilon^{2c},$$
which then implies that, for any $t>0$,
\begin{flalign*}
&\Pr{f\notin\{-1/2,1/2\}+t[-\epsilon^{c},\epsilon^{c}]}\\
&\qquad\leq\Pr{|f-a_{i}x_{i}|>t\epsilon^{c}}\leq\|f-a_{i}x_{i}\|^{2}/[t^{2}\epsilon^{2c}]\leq36/t^{2}.
\end{flalign*}
But this violates \eqref{ten2} when $t=7$, so that $\epsilon^{2(1-c)}<13/49$ (and $(t+1)\epsilon^{c}<1$, i.e. $\epsilon^{c}<1/8$), i.e. this cannot occur for small $\epsilon$.  Similarly, with exactly three indices $a_{i},a_{j},a_{k}$ near $1/2$, we have $\ell(x)\colonequals a_{i}x_{i}+a_{j}x_{j}+a_{k}x_{k}$ and
\begin{flalign*}
&\Pr{f\notin\{-3/2, -1/2,1/2, 3/2\}+t[-3\epsilon^{c},3\epsilon^{c}]}\\
&\qquad\leq\Pr{|f-\ell(x)|>3t\epsilon^{c}}\leq\|f-\ell(x)\|^{2}/[\epsilon^{2c}9t^{2}]=4/t^{2}.
\end{flalign*}
But this violates \eqref{ten2} when $t=2.1$, so that $\epsilon^{2(1-c)}<.41/4.41$ (and $(3t+1)\epsilon^{c}<1/2$, i.e. $\epsilon^{c}<1/(2\cdot7.3)$), i.e. we cannot have exactly three indices of this form.

The remaining case of exactly four indices satisfying $\abs{a_{i}}\in 
1/2+[-\epsilon^{c},\epsilon^{c}]$ cannot occur for $\epsilon$ small, but we do not need to rule this out, since it will not improve our bounds later on. 

\textbf{Combining all cases}

From the Pythagorean Theorem in the form $\vnorm{f}^{2}=\vnorm{f-\ell}^{2}+\vnorm{\ell}^{2}$, where $\ell$ is a linear function containing some coefficients of $f$ (depending on the above case under consideration), if $\epsilon<(1/2)^{1/(2(1-c))}$, then
\begin{itemize}
\item[(1)] $\vnorm{f}^{2}\geq (1-\epsilon^{c})^{2}  = 1-2\epsilon^{c}+\epsilon^{2c}$, or
\item[(2)] $\vnorm{f}^{2}\leq \epsilon^{2}+36\epsilon^{2c}$, if $\epsilon<.089656^{1/[2(1-c)]}$, or
\item[(3)] $\vnorm{f}^{2}\geq2(1/2 - \epsilon^{c})^{2} - 36\epsilon^{2c}=1/2 - 2\epsilon^{c}+2\epsilon^{2c}-36\epsilon^{2c}$, if $\epsilon<.089656^{1/[2(1-c)]}$, or $\vnorm{f}^{2}\leq 1/2-2\epsilon^{c} + 38\epsilon^{2c}$, or
\item[(4)] 
$\vnorm{f}^{2}\geq4(1/2 - \epsilon^{c})^{2} - 36\epsilon^{2c}=1 - 4\epsilon^{c}+4\epsilon^{2c}-36\epsilon^{2c}$, if $\epsilon^{c}<1/14.6$.
\end{itemize}

Then, from the Pythagorean Theorem in the form $\vnorm{g}^{2}=\vnorm{f-g}^{2}+\vnorm{f}^{2}$, we have: if $\epsilon<.089656^{1/[2(1-c)]}$, and if $\epsilon^{c}<1/14.6$, then
\begin{itemize}
\item[(1)] $\vnorm{g}^{2}\geq (1-\epsilon^{c})^{2}= 1-2\epsilon^{c}+\epsilon^{2c}\geq 1-2\epsilon^{c}$, or
\item[(2)] $\vnorm{g}^{2}\leq 2\epsilon^{2}+36\epsilon^{2c}\leq 38\epsilon^{2c}$, or
\item[(3)] $\vnorm{g}^{2}\geq2(1/2 - \epsilon^{c})^{2} - 36\epsilon^{2c}+\epsilon^{2}\geq1/2 - 2\epsilon^{c}-34\epsilon^{2c}$, or \\ $\vnorm{g}^{2}\leq 1/2 - 2\epsilon^{c}+40\epsilon^{2c}$, or
\item[(4)] 
$\vnorm{g}^{2}\geq4(1/2 - \epsilon^{c})^{2} - 36\epsilon^{2c}+\epsilon^{2}\geq1 - 4\epsilon^{c}-32\epsilon^{2c}$.
\end{itemize}

Substituting the definition of $\epsilon$ from \eqref{epsdef}, and solving for $\epsilon^{2}$ or $\epsilon^{c}$,
\begin{itemize}
\item[(1)] $\vnorm{f-g}^{2}\geq2^{-2/c}(1-\vnorm{g}^{2})^{2/c}$, if $\vnorm{g}^{2}\leq1$.
\item[(2)] $\vnorm{f-g}^{2}\geq\vnorm{g}^{2/c}38^{-1/c}$.
\item[(3)] $\vnorm{f-g}^{c}\geq\frac{\sqrt{4+4(34)(1/2 - \vnorm{g}^{2})} - 2}{68}$, if $\vnorm{g}^{2}\leq 1/2$, or\\
$\vnorm{f-g}^{c}\geq\frac{\sqrt{4-4(40)(1/2 - \vnorm{g}^{2})} - 2}{80}$, if $\vnorm{g}^{2}\geq 1/2$, or
\item[(4)] $\vnorm{f-g}^{c}\geq\frac{\sqrt{16+4(32)(1 - \vnorm{g}^{2})} - 4}{64}$, if $\vnorm{g}^{2}\leq1$.
\end{itemize}

%
%

So, putting everything together, and noting that $\vnorm{g}^{2}=2\Pr{A}$, since $A\cap(-A)=\emptyset$, either (1) through (4) hold, or 
$$
\vnorm{f-g}^{2}\geq\min\Big(.089656^{1/(1-c)},\quad (1/14.6)^{2/c}\Big)
$$

So, choosing $c=.69$, we get, either $\vnorm{f-g}^{2}\geq .000422$, or

\begin{flalign*}
\vnorm{f-g}^{2}
\geq\min\Big(
&
2^{-2/.69}(1-2\Pr{A})^{2/.69}\ind{\Pr{A}<1/2}
,\quad 
(2\Pr{A})^{1/.69}38^{-1/.69}
,\\
&\qquad
\Big[\frac{\sqrt{4+4(34)(1/2 - 2\Pr{A})} - 2}{68}\Big]^{2/.69}\ind{\Pr{A}<1/4}
\\
&\qquad\quad+\Big[\frac{\sqrt{4-4(40)(1/2 - 2\Pr{A})} - 2}{80}\Big]^{2/.69}\ind{\Pr{A}\geq1/4}
,\\
&\qquad
\Big[\frac{\sqrt{16+4(32)(1 - 2\Pr{A})} - 4}{64}\Big]^{2/.69}\ind{\Pr{A}<1/2}
\Big)
\end{flalign*}

Since

\begin{flalign*}
\vnorm{f - g}^{2}
&=\vnorm{f-(\ind{A}-\ind{-A})}^{2}
=4\vnorm{f/2 - (\ind{A}-\ind{-A})/2}^{2}\\
&=4\sum_{\abs{S}\geq3\,\text{odd}}\absf{\widehat{\ind{A}}(S)}^{2}.
\end{flalign*}
We finally conclude the lemma, noting also that the first term in the minimum is an upper bound for the last term, i.e. we can remove the first term from the minimum.

\end{proof}

We used an assumption that $A=-A$ in the previous Lemma, which is justified by the following structural result for Chang's inequality maximizers.

\begin{lemma}\label{lemma8}
Let $A\subset\{-1,1\}^{n}$ and let $\alpha\colonequals \PP(A)$.  Let $z\colonequals \E{X 1_{A}(X)}\in\R^{n}$.  Suppose $A$ maximizes $\sum_{i=1}^{n}\absf{\widehat{1_{B}}(i)}^{2}$ subject to the constraint $\PP(B)=\alpha$ over all $B\subset\{-1,1\}^{n}$.  Then $A$ is a half space, i.e. $\exists$ $u\in\R$ such that
$$A=\{x\in\{-1,1\}^{n}\colon\langle x,z\rangle\geq u\}.$$
\end{lemma}
\begin{proof}
Let $t\in\R$ and $S_{t}\colonequals\{x\in\{-1,1\}^{n}\colon\langle x,z\rangle\geq t\}$.  Let $u\colonequals\inf\{t\in\R\colon \PP(S_{t})<\PP(A)\}$.  We will first show that, for any $\epsilon>0$,
$$S_{u+\epsilon}\subset A.$$
Assume for the sake of contradiction that this does not occur.  Then there exists $x\in A,y\notin A$ and $t\in\R$ such that $\langle z,x\rangle<t$ and $\langle z,y\rangle\geq t$.  Consider the set $A'$ defined by $A'\colonequals (A\setminus\{x\})\cup\{y\}$.  Denote $z'\colonequals \E{X 1_{A'}(X)}\in\R^{n}$.  Then
\begin{flalign*}
\sum_{i=1}^{n}\absf{\widehat{1_{A'}}(i)}^{2}
-\sum_{i=1}^{n}\absf{\widehat{1_{A}}(i)}^{2}
&=\vnorm{z'}^{2}-\vnorm{z}^{2}\\
&=\vnorm{z-x2^{-n}+y2^{-n}}^{2}-\vnorm{z}^{2}\\
&=2\cdot 2^{-n}\langle z, y-x\rangle+ 2^{-2n}\vnorm{y-x}^{2}.
\end{flalign*}
By definition of $x,y,t$, we have $\langle z, y-x\rangle\geq0$.  So,
$$
\sum_{i=1}^{n}\absf{\widehat{1_{A'}}(i)}^{2}
-\sum_{i=1}^{n}\absf{\widehat{1_{A}}(i)}^{2}
>0.
$$
We have achieved a contradiction.  We conclude that $A\supset S_{u+\epsilon}$.

We now consider what happens when $\epsilon=0$.  By definition of $S_{u}$, we have $\PP(S_{u+\epsilon})< \PP(A)\leq \PP(S_{u})$.  If $\PP(S_{u})=\PP(A)$, we conclude that $A=S_{u}$, and the proof is completed.  In the remaining case that $\PP(S_{u})>\PP(A)$, the cardinality of $\partial S_{u}\colonequals\{x\in\{-1,1\}^{n}\colon\langle x,z\rangle=u\}$ is larger than one.  So, let $x,y\in\partial S_{u}$ with $x\in A$ and $y\notin A$.  Repeating the above argument, we arrive at a contradiction.  We conclude that $A=S_{u}$, as desired.
\end{proof}

\begin{proof}[Proof of Improvement to Theorem \ref{thm:main}]
Combining Lemma \ref{newlemma} with \eqref{bonus} and Lemma \ref{changlemma} we have: if $h\colon\{-1,1\}^{n}\to\{0,1\}$, and if $\alpha\colonequals\Pr{h(X)=1}$, then
\begin{equation}\label{bigbound}
\begin{aligned}
\sum_{k=1}^{n}|\widehat{h}(k)|^{2}&
\leq\min\Big[2\alpha^{2}\log(1/\alpha), \\
&\frac{\alpha}{2}
-\frac{1}{4}\min\Big(
(2\Pr{A})^{1/.69}38^{-1/.69}
,\\
&\qquad
\Big[\frac{\sqrt{4+4(34)(1/2 - 2\Pr{A})} - 2}{68}\Big]^{2/.69}1_{\Pr{A}<1/4}
\\
&\qquad\quad+\Big[\frac{\sqrt{4-4(40)(1/2 - 2\Pr{A})} - 2}{80}\Big]^{2/.69}1_{\Pr{A}\geq1/4}
,\\
&\qquad
\Big[\frac{\sqrt{16+4(32)(1 - 2\Pr{A})} - 4}{64}\Big]^{2/.69}1_{\Pr{A}<1/2}
\Big)\Big]
\end{aligned}
\end{equation}
The last quantity is depicted as the dotted line in Figure \ref{figone}.  

Let $f,g\colon\{-1,1\}^{n}\to[n]$.  Let $\Omega_{f}\subseteq[n]$ be 6 indices (or $n$ indices if $n<6$) $i\in[n]$ with the largest values of $\Pr{f(X)=i}$.  Similarly define $\Omega_{g}\subseteq[n]$.

Suppose for now that $\Pr{g(X)\notin\Omega_{g}}<.0168995$.  After applying a permutation to $[n]$, we may assume that $\Pr{g(X)>6}<.01270673$.  Then we can apply Lemma \ref{lemma7} with $k=7$ and $\epsilon=.01270673$ to get
$$\E{X_{g(Y)}Y_{f(X)}}\leq\frac{1}{2}-(1-4(.01270673))2^{-6}=0.485169170625.$$
As discussed after the statement of Theorem \ref{thm:main}, this proves Theorem \ref{thm:main} with the bound
$$U = (1+.485169173)/4\leq 0.37129229325.$$
We get the same conclusion from Lemma \ref{lemma7} if $\Pr{f(X)\notin\Omega_{f}}<.01270673$.  The only remaining case to consider is:
$$\Pr{f(X)\notin\Omega_{f}}\geq .01270673\quad\text{and}\quad\Pr{g(X)\notin\Omega_{g}}\geq .01270673.$$
In this case, as in Lemma \ref{lemma6} it suffices to upper bound $\|F\|^{2}$ by $.485169173$.  It then suffices to maximize the sum of the right side of \eqref{bigbound}, as in \eqref{meq}, subject to the constraint
$$\sum_{j=7}^{n}\alpha_{j}\geq.01270673,
\qquad \alpha_{1}\geq\cdots\geq\alpha_{n},
\qquad \sum_{j=1}^{n}\alpha_{j}=1.$$
These constraints imply $.01270673\leq \alpha_{j}\leq1/5$ for all $5\leq j\leq 7$.  In this range of values, the smallest value of the right side of \eqref{bigbound} occurs when $\alpha_{j}=.01270673$ for all $5\leq j\leq 7$ (which can be verified numerically).  Also, this value occurs where the right side of \eqref{bigbound} is equal to $2\alpha_{j}^{2}\log(1/\alpha_{j})$ for each $5\leq j\leq 7$. 
If $\alpha_{1},\alpha_{2},\ldots$ maximizes a sum of the terms on the right of \eqref{bigbound} subject to these constraints, we then get an upper bound of the form
$$
\max_{s\in[.01270673,1/5]}
\Big(
(1-3s)/2 
+3\cdot 2 s^{2}\log(1/s)
\Big)
$$
which is achieved as $s=.0168995$ with value
$$
(1-5(.01270673))/2 
+5\cdot 2 (.01270673)^{2}\log(1/.01270673)\leq0.485169173.
$$

%

We therefore improved Theorem \ref{thm:main} from a bound of $.37406$ to a bound of $.37193$.

%
%
\end{proof}









\end{document}

%% file: changimprove.pdf_tex
\begingroup%
  \makeatletter%
  \providecommand\color[2][]{%
    \errmessage{(Inkscape) Color is used for the text in Inkscape, but the package 'color.sty' is not loaded}%
    \renewcommand\color[2][]{}%
  }%
  \providecommand\transparent[1]{%
    \errmessage{(Inkscape) Transparency is used (non-zero) for the text in Inkscape, but the package 'transparent.sty' is not loaded}%
    \renewcommand\transparent[1]{}%
  }%
  \providecommand\rotatebox[2]{#2}%
  \newcommand*\fsize{\dimexpr\f@size pt\relax}%
  \newcommand*\lineheight[1]{\fontsize{\fsize}{#1\fsize}\selectfont}%
  \ifx\svgwidth\undefined%
    \setlength{\unitlength}{364.40270612bp}%
    \ifx\svgscale\undefined%
      \relax%
    \else%
      \setlength{\unitlength}{\unitlength * \real{\svgscale}}%
    \fi%
  \else%
    \setlength{\unitlength}{\svgwidth}%
  \fi%
  \global\let\svgwidth\undefined%
  \global\let\svgscale\undefined%
  \makeatother%
  \begin{picture}(1,0.5377573)%
    \lineheight{1}%
    \setlength\tabcolsep{0pt}%
    \put(0,0){\includegraphics[width=\unitlength,page=1]{changimprove.pdf}}%
    \put(0.48898676,0.06053364){\color[rgb]{0,0,0}\makebox(0,0)[lt]{\lineheight{1.25}\smash{\begin{tabular}[t]{l}$1/4$\end{tabular}}}}%
    \put(0.38096661,0.05922961){\color[rgb]{0,0,0}\makebox(0,0)[lt]{\lineheight{1.25}\smash{\begin{tabular}[t]{l}$.1161$\end{tabular}}}}%
    \put(0.78294908,0.05812823){\color[rgb]{0,0,0}\makebox(0,0)[lt]{\lineheight{1.25}\smash{\begin{tabular}[t]{l}$1/2$\end{tabular}}}}%
    \put(0.15638764,0.2158589){\color[rgb]{0,0,0}\makebox(0,0)[lt]{\lineheight{1.25}\smash{\begin{tabular}[t]{l}$2x^{2}\log(x)$\end{tabular}}}}%
    \put(0.53411255,0.44761213){\color[rgb]{0,0,0}\makebox(0,0)[lt]{\lineheight{1.25}\smash{\begin{tabular}[t]{l}$x/2$\end{tabular}}}}%
    \put(0,0){\includegraphics[width=\unitlength,page=2]{changimprove.pdf}}%
    \put(0.95203137,0.10457081){\color[rgb]{0,0,0}\makebox(0,0)[lt]{\lineheight{1.25}\smash{\begin{tabular}[t]{l}$x$\end{tabular}}}}%
  \end{picture}%
\endgroup%